\documentclass{amsart}

\title[Ubiquity in forward-reverse setting]{Ubiquity in the setting of forward-reverse Brascamp--Lieb inequalities}
\author{Mirei Watanabe}
\date{October 22, 2026}

\usepackage{amsfonts, amssymb, amsmath, bm}
\usepackage{xcolor}
\numberwithin{equation}{section}

\newtheorem{theorem}{Theorem}[section]
\newtheorem{lemma}[theorem]{Lemma}
\newtheorem{proposition}[theorem]{Proposition}

\theoremstyle{definition}
\newtheorem{definition}[theorem]{Definition}

\begin{document}

\begin{abstract}
The ubiquity of geometric Brascamp--Lieb data, which refers to a certain density property of geometric data within the set of all feasible Brascamp--Lieb data, has recently been studied. Motivated in part by a question of Bennett--Tao concerning a possible connection or analogue of the adjoint Brascamp--Lieb inequality to the forward-reverse Brascamp--Lieb inequality, we establish ubiquity in the forward-reverse setting. Our proof takes an analytic approach and, to overcome the difficulties arising in extending ubiquity to the forward-reverse setting, makes use of tools from the calculus of variations such as Danskin's theorem and Ekeland's variational principle.

\end{abstract}

\maketitle

\section{Introduction}

The Brascamp--Lieb inequality has been studied by many authors, and a rich theory about the finiteness and properties of optimal constants have been obtained  (see, for example \cite{Bar}, \cite{BCCT2}, \cite{BL} and \cite{BCCT}). A remarkable theorem of Lieb \cite{Lieb} says that each optimal constant is saturated by centered gaussians. Similarly, the reverse Brascamp–Lieb inequality has also been studied extensively, beginning with the work of Barthe \cite{Bar}. Recently, the theory has been extensively advanced to the forward-reverse setting, which recovers both the classical and reverse forms in work of Liu--Courtade--Cuff--Verdú \cite{LCCV} and Courtade--Liu \cite{CL}. 

Beginning with its application to the adjoint Brascamp--Lieb inequality \cite{BT}, the property that so-called geometric data are dense, in a suitable sense, in the space of all feasible data has been studied. This was first proved in \cite{BGTP} and answered Question 10.6 in the paper \cite{BT} by Bennett and Tao on the adjoint Brascamp--Lieb inequality (an alternative proof was later given by Chindris--Derksen \cite{CDN}). Geometric data are relatively tractable, and such density results, referred to as ubiquity in \cite{BGTP}, are expected to be useful in reducing certain problems for general feasible data to the corresponding problems for geometric data.

In \cite{BT}, Bennett and Tao also raised the question of whether there is any relation or analogue between the adjoint Brascamp--Lieb inequalities and Barthe's reverse Brascamp--Lieb inequality, or, more generally, the forward-reverse Brascamp--Lieb inequality. In view of the role played by ubiquity in the study of the adjoint Brascamp--Lieb inequality, establishing an analogous ubiquity result in the forward-reverse setting may provide a useful step towards understanding such a connection. In this paper, we establish the ubiquity of geometric Brascamp--Lieb data in the setting of forward-reverse Brascamp--Lieb inequalities, thereby extending this structural principle to a substantially broader framework.

\subsection{Forward-reverse Brascamp--Lieb inequalities}
Let $k$ and $m$ be positive integers. Let $\mathbf{c}=(c_1,\ldots,c_k)$ and $\mathbf{d}=(d_1,\ldots,d_m)$ be couples of positive reals and let each $E_i$ and $E^j$ be Euclidean spaces such that \[\sum^k_{i=1}c_i\dim E_i=\sum^m_{j=1}d_j\dim E^j.\] Set $E_0:=\bigoplus^k_{i=1}E_i$. For each $i\in[k]$, we denote by $\pi_i:E_0\to E_i$ the orthogonal projection. Let $\mathbf{B}=(B_{ij}:E_i\to E^j)_{i\in[k],j\in[m]}$ be a collection of linear maps. The triple $(\mathbf{c},\mathbf{d},\mathbf{B})$ is called a \textit{Brascamp--Lieb datum}. The corresponding forward-reverse Brascamp--Lieb inequality takes the form 
\begin{equation}\label{eq:BL}
    \prod^k_{i=1}\bigg(\int_{E_i} f_i\bigg)^{c_i}\le C \prod^m_{j=1}\bigg(\int_{E^j} g_j\bigg)^{d_j}
\end{equation}
under the condition that the nonnegative and measurable functions $f_i$ and $g_j$ satisfy
\begin{equation}\label{condition}
    \prod^k_{i=1}f^{c_i}_i(\pi_i(\bm{x}))\le \prod^m_{j=1}g^{d_j}_j\big(\sum^k_{i=1}B_{ij}\pi_i(\bm{x})\big) \qquad (\bm{x}\in E_0).
\end{equation}
We denote by $\mathrm{BL}({\mathbf{c}},\mathbf{d},\mathbf{B})$ the optimal constant in \eqref{eq:BL}, and by $\mathrm{BL_g}({\mathbf{c}},\mathbf{d},\mathbf{B})$ the optimal constant in \eqref{eq:BL} when $f_i(x)$ and $g_j(x)$ are centred gaussians i.e. for $i\in[k]$ and $j\in[m]$, $f_i(x)=\exp{(-\pi\langle V_ix,x\rangle)}$ and $g_j(x)=\exp{(-\pi\langle U_jx,x\rangle)}$ for some $V_i\in\mathsf{PD}_{\dim{E_i}}$ and $U_j\in\mathsf{PD}_{\dim{E^j}}$. Here $\mathsf{PD}_{n}$ denotes the set of all real, self-adjoint and positive definite $n\times n$ matrices. Similarly, $\mathsf{ND}_{n}$ denotes the set of all real, self-adjoint and nonnegative definite $n\times n$ matrices. We say that $({\mathbf{c}},\mathbf{d},\mathbf{B})$ is \textit{feasible} when $\mathrm{BL}({\mathbf{c}},\mathbf{d},\mathbf{B})<\infty$. 

\medskip
\noindent
\textbf{Remark.} When $k=1$ and $c_1=1$, \eqref{eq:BL} is the classical Brascamp--Lieb inequality studied in, for example, \cite{BCCT2,BCCT,BL}. Also when $m=1$ and $d_1=1$, \eqref{eq:BL} is Barthe's reverse Brascamp--Lieb inequality \cite{Bar}.

\medskip

From now on, we denote $\Lambda_{\mathbf{c}}:=\bigoplus^k_{i=1}c_i\mathrm{id}_{E_i}$, $B_j:=\sum^k_{i=1}B_{ij}\circ \pi_{E_i}$, $B:=\sum^m_{j=1}B_j$,  $V_{\mathbf{c}}:=\bigoplus^k_{i=1}c_iV_i$ and $U_{\mathbf{d}}:=\bigoplus^m_{j=1}d_jU_i$.

\medskip
\noindent
\textbf{Remark.} Substituting $f_i(x)=\exp{(-\pi\langle V_ix,x\rangle)}$ and $g_j(x)=\exp{(-\pi\langle U_ix,x\rangle)}$, we see that \eqref{condition} is represented by
\begin{equation}\label{satucondition}
    B^*U_{\mathbf{d}}B\le V_{\mathbf{c}}.
\end{equation}
Also, \eqref{eq:BL} is represented by
\begin{equation}\label{satuBL}
    \prod^m_{j=1}(\det U_j)^{d_j}\le C\prod^k_{i=1}(\det V_i)^{c_i}.
\end{equation}

For $K_i\in\mathsf{PD}_{\dim E_i}$, $i\in[k]$, we write $\Pi(K_1,\ldots,K_k)$ for the associated coupling i.e. the set of all $K\in\mathsf{ND}_{\dim E_0}$ such that, for all $i\in[k]$,
    \[\langle Kx_i,x_i\rangle=\langle K_ix_i,x_i\rangle, \quad x_i\in E_i.\]

\medskip
\noindent
Remarkably, Lieb established the gaussian saturation of Brascamp--Lieb constants in the case $k=1$ and $c_1=1$ \cite{Lieb} and later Barthe proved an analogue when $m=1$ and $d_1=1$ \cite{Bar}. In \cite{LCCV}, Liu, Courtade, Cuff and Verdú proved a unified version in the full generality of the forward-reverse setting (see also \cite{CL}).

\begin{theorem}[\cite{LCCV}]\label{thm:saturation}
Let $({\mathbf{c}},\mathbf{d},\mathbf{B})$ be a Brascamp--Lieb datum. Then
\begin{equation}
    \mathrm{BL}({\mathbf{c}},\mathbf{d},\mathbf{B})=\mathrm{BL_g}({\mathbf{c}},\mathbf{d},\mathbf{B}).
\end{equation}

\end{theorem}

\noindent
Also in \cite[Theorem 1.3]{CL}, Courtade and Liu give the following duality. Theorem \ref{thm:redusing} allows us to express, in a unified formula, a class of functions closely related to the Brascamp--Lieb constant, which will be defined later.

\begin{theorem}[\cite{CL}]\label{thm:redusing}
    Fix $-\infty <C\le \infty$. The following statements are equivalent. 
    
\medskip
\noindent
 1. For all $V_i\in\mathsf{PD}_{\dim{E_i}}$, $i\in[k]$ and $U_j\in\mathsf{PD}_{\dim{E^j}}$, $j\in[m]$ satisfying \eqref{satucondition}, 
\begin{equation}
   \sum^m_{j=1}d_j\log\det U_j \le C +\sum^k_{i=1}c_i\log\det V_i.
\end{equation}
   
\medskip
\noindent
2.  For all $K_i\in\mathsf{PD}_{\dim{E_i}}$, $i\in[k]$,
\begin{equation}\label{eq;reduct-coupling}
     \sum^k_{i=1}c_i\log\det K_i\le C+ \max_{K\in\Pi(K_1.\ldots,K_k)}\sum^m_{j=1}d_j\log\det (B_jKB_j^*).
\end{equation}
\end{theorem}

\medskip
\noindent
\textbf{Remark.} In the classical Brascamp--Lieb setting, corresponding to $k=1$ and $c_1=1$, \eqref{eq;reduct-coupling} can be written in the following form:
\begin{equation}
    \log\det K_1\le C+ \sum^m_{j=1}d_j\log\det (B_jK_1B_j^*).
\end{equation}
This reformulation also appears in \cite{BT,CCE}. In this case, there is no need to take a maximum over couplings. In this sense, the forward-reverse setting introduces an additional significant difficulty that is absent in the classical Brascamp--Lieb case.

\subsection{Geometric data and main result}
Let us introduce geometric data, which are Brascamp–Lieb data with special properties. 
\begin{definition}
    We say that a Brascamp--Lieb datum $(\mathbf{c},\mathbf{d},\mathbf{B})$ is \textit{geometric} when there exists some $\Pi\in\Pi(\mathrm{id}_{E_1},\ldots,\mathrm{id}_{E_k})$ such that 
    \begin{equation}
        B_j\Pi B_j^*=\mathrm{id}_{E^j}\quad\text{for all $j\in[m]$}
    \end{equation}
    and
    \begin{equation}
        \sum^m_{j=1}d_jB_j^*B_j\le \Lambda_{\mathbf{c}}.
    \end{equation}
\end{definition}

\noindent
Geometric data have the very nice feature that $\mathrm{BL}({\mathbf{c}},\mathbf{d},\mathbf{B})=1$. For $k=1$ and $c_1=1$, this property was first proved by Ball (rank-one) \cite{Ball} and subsequently by Barthe in the general-rank case; Barthe also established the corresponding result for $m=1$ and $d_1=1$ \cite{Bar}. (See also \cite[Theorem 2.8]{BCCT}.) Also, Courtade and Liu established this fact in the context of forward-reverse Brascamp--Lieb inequalities \cite{CL}. In general, it is difficult to compute $\mathrm{BL}({\mathbf{c}},\mathbf{d},\mathbf{B})$ and so geometric data are easier to handle in that respect.

\begin{definition}
    We say that two tuples of linear maps $\mathbf{B}=(B_{ij})$ and $\mathbf{B}'=(B_{ij}')$ are \textit{equivalent} if there exist invertible matrices $S_i\in\mathbb{R}^{\dim E_i\times\dim E_i}$ and $T_j\in\mathbb{R}^{\dim E^j\times\dim E^j}$ ($i\in[k]$, $j\in[m]$) such that for every $j\in[m]$,
    \[B'_j=T_j^{-1}B_j\big(\bigoplus^k_{i=1} S_i\big).\]
    \end{definition}

\medskip
\noindent
We recall the relation between gaussian extremizability and geometric data. A datum $(\mathbf{c},\mathbf{d},\mathbf{B})$ is called gaussian extremizable if there exists a gaussian input satisfying \eqref{satucondition} for which equality is attained in \eqref{satuBL}. In particular, every gaussian extremizable datum is equivalent to a geometric datum (see for example \cite{BCCT,CL}). Also, feasibilty is invariant under the equivalence. We note that the definition of equivalent data in \cite[Definition 1.22]{CL} seems to be slightly different to ours in the sense that the former appears to allow general invertible maps on $\mathbb{R}^{\dim E_0\times\dim E_0}$ rather than those which have a block diagonal structure. Examples show that feasibility is not preserved if general invertible maps are permitted.

\medskip
Now we introduce our main theorem. 

\begin{theorem}\label{mainthm:ubiquity}
    Let $(\mathbf{c}, \mathbf{d},\mathbf{B})$ be a feasible datum. Then there exists some tuple of linear maps $\mathbf{G}$ such that $(\mathbf{c}, \mathbf{d},\mathbf{G})$ is geometric and for any $\varepsilon>0$, there is some tuple of linear maps $\mathbf{B}'$ which is equivalent to $\mathbf{B}$ such that \[\| \mathbf{G}-\mathbf{B}' \|< \varepsilon.\]
\end{theorem}

\medskip
\noindent
\textbf{Remark.} The set of gaussian extremizable data is not dense in the set of feasible data; we refer the reader to \cite{BGTP} for a counterexample.

\medskip
In the case of $k=1$ and $c_1=1$, Bez, Gauvan and Tsuji \cite{BGTP} proved this theorem using operator scalings \cite{GGOW}, an algorithmic method. In \cite{CDN}, Chindris and Derksen also proved an extension to quiver data. They proved it by algebraic methods via a Jordan--Holder filtration. 

Inspired by work of Dvir--Hu \cite{DH}, we adopt an analytic approach to establish ubiquity in the forward-reverse Brascamp--Lieb setting. In the classical Brascamp--Lieb case, as well as for capacity associated with quiver data, the argument can be carried out essentially by relying on the proof of \cite[Theorem 1.4]{DH}. The forward-reverse setting, however, presents an additional difficulty; as can be seen from Theorem \ref{thm:redusing}, one has to deal with a maximization over coupling matrices, whose maximizers do not in general admit an explicit representation. To overcome this difficulty, we use Danskin's theorem to establish the required differentiability and Ekeland's variational principle to obtain suitable estimates for the derivatives. 

Our main result is stated in Section 2 and the proof is given in Section 3. Finally we include an Appendix with a proof of ubiquity in the setting of quiver capacity. This provides a new proof \cite[Corollary 14]{CDN} and helps to illustrate our comment that the forward-reverse Brascamp--Lieb setting requires new argument that, to the best our knowledge, have not been used in the theory of Brascamp--Lieb inequalities before.

\section{Preliminaries and main result}

To prove Theorem \ref{mainthm:ubiquity}, we establish the following property.

\begin{theorem}\label{thm:Main}
  Let  $(\mathbf{c},\mathbf{d},\mathbf{B})$ be a feasible datum. Then for each $\varepsilon>0$, there exists a tuple of linear maps $\mathbf{B}'=(B'_{ij})$ equivalent to $\mathbf{B}$ and some $\Pi\in\Pi(\mathrm{id}_{E_1},\ldots,\mathrm{id}_{E_k})$ such that 
  \begin{equation}\label{eq:seqarg}
       \|B_j'\Pi B_j'^*-\mathrm{id}_{E^j}\|<\varepsilon\quad\text{for all $j\in[m]$}
  \end{equation}
  and
  \begin{equation}\label{eq:seqproj}
      \sum^m_{j=1}d_jB_j'^*B_j'\le \Lambda_{\mathbf{c}}.
  \end{equation}
\end{theorem}

\begin{proof}[Proof of Theorem \ref{mainthm:ubiquity}]

Suppose that Theorem \ref{thm:Main} holds. Choose two sequences $(\mathbf{B}^{(n)})_{n\in\mathbb{N}}$ and $(\Pi^{(n)})_{n\in\mathbb{N}}$ satisfying \eqref{eq:seqarg} and \eqref{eq:seqproj} in the case of $\varepsilon=1/n$. Since both two sets
\[\{\mathbf{B}=(B_{ij}): \sum^m_{j=1}d_jB_j^*B_j \le\Lambda_{\mathbf{c}}\}\text{ and }\Pi(\mathrm{id}_{E_1},\ldots,\mathrm{id}_{E_k})\] are compact, there exists a subsequence $(n_l)\subset(n)$, a matrix $\Pi_\mathbf{G}\in\Pi(\mathrm{id}_{E_1},\ldots,\mathrm{id}_{E_k})$ and a tuple of linear maps $\mathbf{G}=(G_{ij})$ such that \[\lim_{l\to\infty}\|\mathbf{B}^{(n_l)}-\mathbf{G}\|=0, \lim_{l\to\infty}\|\Pi_\mathbf{G}-\Pi^{(n_{l})}\|=0 \text{ and }\sum^m_{j=1}d_jG_j^*G_j \le\Lambda_{\mathbf{c}}. \]  By \eqref{eq:seqarg}, we obtain
\[G_j\Pi_\mathbf{G}G_j^*=\lim_{l\to\infty}B^{(n_{l})}_j\Pi^{(n_{l})}(B^{(n_{l})}_j)^*=\mathrm{id}_{E^j}.\] This means that the triple $(\mathbf{c},\mathbf{d},\mathbf{G})$ is geometric.
\end{proof}

We give a proof of Theorem \ref{thm:Main} using an analytic approach. We first construct a differentiable function that is bounded from above and associated with a feasible Brascamp--Lieb datum. In Section 3, we prove Theorem \ref{thm:Main} by the analysis of the function.

\medskip
Using Theorems \ref{thm:saturation} and \ref{thm:redusing}, we see that the function $\Phi:\prod^k_{i=1}\mathsf{ND}_{\dim{E_i}}\to\mathbb{R}$, 
\begin{align*}
    \Phi(K_1,\ldots,K_k)&=\sum^k_{i=1}c_i\log\det K_i-\max_{K\in\Pi(K_1.\ldots,K_k)}\sum^m_{j=1}d_j\log\det (B_jKB_j^*)
\end{align*}
 satisfies $\Phi=-\infty$ on $\prod^k_{i=1}\mathsf{ND}_{\dim{E_i}}-\prod^k_{i=1}\mathsf{PD}_{\dim{E_i}}$ and
\begin{equation}
   2\log \mathrm{BL}({\mathbf{c}},\mathbf{d},\mathbf{B})=\sup_{K_1,\ldots,K_k}\Phi(K_1,\ldots,K_k).
\end{equation}
Thus, $\mathrm{BL}({\mathbf{c}},\mathbf{d},\mathbf{B})<\infty$ if and only if $\Phi$ is bounded from above. However, we do not know whether $\Phi$ is differentiable. To resolve this, we define the following function. Let $\tau>0$ and define $\Phi_\tau:\prod^k_{i=1}\mathsf{ND}_{\dim{E_i}}\to\mathbb{R}$ to be the function given by \[\Phi_{\tau}(K_1,\ldots,K_k)=\sum^k_{i=1}c_i\log\det K_i-H_\tau(K_1,\ldots,K_k).\]Here \[H_\tau(K_1,\ldots,K_k)=\max_{\Pi\in\Pi(\mathrm{id}_{E_1}.\ldots,\mathrm{id}_{E_k})}\bigg(\sum^m_{j=1}d_j\log\det (B_j\tilde{K}^*\Pi \tilde{K}B_j^*)-\tau {\| \Pi}\|^2 \bigg),\] 
where $\tilde{K}$ is a matrix such that $\tilde{K}=\bigoplus_i\tilde{K}_i$ with $\tilde{K}_i^*\tilde{K}_i= K_i$. We consider the norm $\|\cdot\|$ as the Frobenius norm in concrete calculations.

 \begin{proposition}
 The function $\Phi_\tau$ satisfies the following.
 
 \medskip
 \noindent
 1. $\Phi_\tau=-\infty$ on $\prod^k_{i=1}\mathsf{ND}_{\dim{E_i}}-\prod^k_{i=1}\mathsf{PD}_{\dim{E_i}}$ and
 \begin{equation}\label{eq:tausaturation}
     2\log \mathrm{BL}({\mathbf{c}},\mathbf{d},\mathbf{B})=\lim_{\tau\to0}\sup_{K_1,\ldots,K_k}\Phi_{\tau}(K_1,\ldots,K_k).
 \end{equation}In particular, $\Phi_\tau$ is bounded from above if and only $\mathrm{BL}({\mathbf{c}},\mathbf{d},\mathbf{B})<\infty$.

 \medskip
 \noindent
 2. $\Phi_{\tau}$ is differentiable for all $i\in[k]$, that is, for any differentiable map $s\mapsto K(s)$, there exists the derivative
 \begin{equation}
    \frac{d}{ds}\Phi_{\tau}(K_1,\ldots,K(s),\ldots,K_k).
 \end{equation}
\end{proposition}

\begin{proof}
 First, by a simple calculation, we have
 \begin{align*}
      \Phi(K_1,\ldots,K_k)+C_1\tau&\ge \Phi_{\tau}(K_1,\ldots,K_k)\ge\Phi(K_1,\ldots,K_k)+C_2\tau
 \end{align*}
     for each $(K_1,\ldots,K_k)$ and $\tau>0$. Hence \eqref{eq:tausaturation} follows.

    Second, we prove the differentiability. Since \[\Pi\mapsto\sum^m_{j=1}d_j\log\det (B_j\tilde{K}^*\Pi \tilde{K}B_j^*)-\tau\|\Pi\|^2\]is strictly concave in $\Pi(\mathrm{id}_{E_1}.\ldots,\mathrm{id}_{E_k})$, we see that there exists a unique $\Pi_{\tau}\in\Pi(\mathrm{id}_{E_1}.\ldots,\mathrm{id}_{E_k})$ such that 
\[H_\tau(K_1,\ldots,K_k)=\sum^m_{j=1}d_j\log\det (B_j\tilde{K}^*\Pi_{\tau} \tilde{K}B_j^*)-\tau\|\Pi_\tau\|^2.\]
Thus, by Danskin's theorem \cite{Dan}, $H_\tau(K_1,\ldots,K_k)$ is differentiable, that is, for each $i$ and differential map $\tilde{K}_i=\tilde{K}_i(s)$,
\begin{equation}\label{eq:differential}
    \frac{d}{ds}H_{\tau}(K_1,\ldots,K_i(s),\ldots,K_k)=\frac{d}{ds}\bigg(\sum^m_{j=1}d_j\log\det (B_j\tilde{K}_i^*(s)\Pi_\tau \tilde{K}_i(s)B_j^*)\bigg).
\end{equation}

\end{proof}

\section{Proof of the main theorem}
\subsection{Structure of $H_{\tau}$}
Fix $\tau>0$ and matrices $K_1,\ldots,K_k$, $K=\bigoplus K_i$. Let $\Pi_\tau$ be a maximizer of $H_\tau(K_1,\ldots,K_k)$. Here we think about an equation which $\Pi_\tau$ satisfies.

\begin{proposition}\label{prop:laglange}
   Fix $\tau>0$ and $K_1,\ldots,K_k$. There exists a tuple of $\dim E_i\times\dim E_i$-matrices $(Z_i)_{i\in[k]}$ such that \begin{equation}\label{eq:Laglange}
        M_{\tau}-2\tau\Pi_{\tau}=Z,
    \end{equation}
    where $Z:=\bigoplus^k_{i=1} Z_i$, $L_j:=B_j\tilde{K}^*\Pi_\tau \tilde{K}B_j^*$ and $M_{\tau}:=\sum^m_{j=1}d_j\tilde{K}B_j^*L^{-1}_jB_j\tilde{K}^*$.
\end{proposition}

\noindent
\textbf{Remark.} By the feasibility of $(\mathbf{c},\mathbf{d},\mathbf{B})$, we see that $\log\det(L_j)>-\infty$ for each $j$. Hence $L_j$ is invertible.

\begin{proof}

 \medskip
 \noindent
 We consider the function
 \[H_{\tau,K}(\Pi)=\sum^m_{j=1}d_j\log\det (B_j\tilde{K}^*\Pi \tilde{K}B_j^*)-\tau\|\Pi\|^2,\] then $\Pi_\tau$ is the maximizer. For $0<t<1$ and $X\in\Pi(\mathrm{id}_{E_1},\ldots,\mathrm{id}_{E_k})$, $(1-t)\Pi_\tau+tX\in \Pi(\mathrm{id}_{E_1},\ldots,\mathrm{id}_{E_k})$. Hence using $\frac{d}{dt}\log\det( A+tB)=\frac{d}{dt}\log\det(I+tA^{-1}B) =\mathrm{tr}(A^{-1}B)$ for invertible matrix $A$, we can calculate as follows.
    \begin{align*}
       0 &=\frac{d}{dt}\bigg|_{t=0}H_{K,\tau}\big((1-t)\Pi_\tau+tX\big)=\frac{d}{dt}\bigg|_{t=0}H_{K,\tau}\big(\Pi_\tau+t(X-\Pi_\tau)\big)\\
        &=\sum^m_{j=1}d_j\mathrm{tr}\big(L^{-1}_jB_j\tilde{K}^*(X-\Pi_\tau)\tilde{K}B_j^*\big)-2\tau \cdot\mathrm{tr}\big(\Pi_\tau  (X-\Pi_\tau)\big)\\
        &=\sum^m_{j=1}d_j\mathrm{tr}\big(\tilde{K}B_j^*L^{-1}_jB_j\tilde{K}^*(X-\Pi_\tau)\big)-2\tau \cdot\mathrm{tr}\big(\Pi_\tau  (X-\Pi_\tau)\big)\\
        &=\mathrm{tr}\big((M_{\tau}-2\tau\Pi_\tau)(X-\Pi_\tau)\big).
    \end{align*}
    Now we substitute $X_1=\mathrm{id}_{E_0}$ and $X_2=\mathrm{id}_{E_0}+\eta (E_{p,q}+E_{q,p})\in \Pi(\mathrm{id}_{E_1},\ldots,\mathrm{id}_{E_k})$
    for any $\dim E_0\times\dim E_0$-matrix unit $E_{p,q}$, whose $(p,q)$ ($p\neq q$) entry is 1, and all other entries are 0, and sufficiently small $\eta>0$. Then we obtain \[\mathrm{tr}\big((M_{\tau}-2\tau\Pi_\tau)(E_{p,q}+E_{q,p})\big)=\eta^{-1}\mathrm{tr}\big((M_{\tau}-2\tau\Pi_\tau)(X_2-X_1)\big)=0.\] Observe that this means all entries in $M_{\tau}-2\tau\Pi_\tau$ are zero except on the $\dim E_i\times\dim E_i$-block diagonal. Thus we obtain  \eqref{eq:Laglange}.
\end{proof}

\medskip
\subsection{Structure of $\Phi_{\tau}$}
We consider the derivative of the original function $\Phi_{\tau}$. As in the work of Dvir and Hu \cite[Lemma 3.3]{DH}, we can control the derivatives in each of the following matrix directions. In the proof, we use Ekeland's variational principle.

\begin{proposition}\label{lem:epsilon}
    Fix $\tau>0$. For each $\varepsilon>0$, there exists $(K_1,\ldots,K_k)\in \prod^k_{i=1}\mathsf{PD}_{\dim(E_i)}$ such that for all $i\in[k]$ and self-adjoint matrix $X_i$ with $\|X_i\|\le1$,
    \begin{equation}\label{eq:variational_ineq}
        \bigg|\frac{d}{ds}\bigg|_{s=0}\Phi_{\tau}(K_1,\ldots,\tilde{K}^*_ie^{sX_i}\tilde{K}_i,\ldots,K_k)\bigg|<\varepsilon.
    \end{equation}
    \end{proposition}

\begin{proof}
    Fix $\varepsilon>0$. Choose a tuple $(K^{(0)}_i)_i$ such that $\Phi_{\tau}((K^{(0)}_i)_i)\ge\sup\Phi_{\tau}-\varepsilon$. Fix $\lambda>0$. Using Ekeland's variational principle \cite{Ekeland} for the function $-\Phi_{\tau}$, there exists a tuple $(K^{\lambda}_i)_i\in \prod^k_{i=1}\mathsf{PD}_{n_i}$  such that $\Phi_{\tau}((K^{\lambda}_i)_i)\ge\Phi_{\tau}((K^{(0)}_i)_i)\ge\sup\Phi_{\tau}-\varepsilon$,
    \begin{equation}\label{eq:nearK_0}
        \|(K^{(0)}_i)_i-(K^{\lambda}_i)_i\|<\lambda
    \end{equation}
    and for any $(K_i)_i\neq (K^{\lambda}_i)_i$,
    \begin{equation}\label{eq:Ekland1}
        \Phi_{\tau}((K_i)_i)<\Phi_{\tau}((K^{\lambda}_i)_i)+\frac{\varepsilon}{\lambda}\|(K_i)_i-(K^{\lambda}_i)_i\|.
    \end{equation}
    Fix $i$ and a self adjoint matrix $X_i$. Considering the Fréchet derivative of $(K^{(s)}_i)_i:=(K^{\lambda}_1,\ldots,(\tilde{K}^\lambda_i)^*e^{sX_i}\tilde{K}^{\lambda}_i,\ldots,K^{\lambda}_k)$, we have
    \begin{equation}\label{eq:varinorm}
       \lim_{s\downarrow0} \frac{\|(K^{(s)}_i)_i-(K^{\lambda}_i)_i\|}{s}= \bigg\|\frac{d}{ds}\bigg|_{s=0}(K^{(s)}_i)_i\bigg\|= \|(\tilde{K}^\lambda_i)^*X_i\tilde{K}^\lambda_i\|\le\|\tilde{K}^\lambda_i\|^2.
    \end{equation}
    Here, we assume that 
    \begin{equation}\label{eq:posotovederi}
    \frac{d}{ds}\bigg|_{s=0}\Phi_{\tau}(K^\lambda_1,\ldots,(\tilde{K}^{\lambda}_i)^*e^{sX_i}\tilde{K}^\lambda_i,\ldots,K^\lambda_k)>0.
    \end{equation}
    By \eqref{eq:nearK_0}, \eqref{eq:Ekland1} and \eqref{eq:varinorm}, we have
    \begin{align*}
        \lim_{s\downarrow0}\bigg|&\frac{\Phi_{\tau}((K^{(s)}_i)_i)-\Phi_{\tau}((K^{\lambda}_i)_i)}{s}\bigg|\\&=\lim_{s\downarrow0}\frac{\Phi_{\tau}((K^{(s)}_i)_i)-\Phi_{\tau}((K^{\lambda}_i)_i)}{s}\\
        &= \lim_{s\downarrow0}\frac{\Phi_{\tau}((K^{(s)}_i)_i)-\Phi_{\tau}((K^{\lambda}_i)_i)}{\|(K^{(s)}_i)_i-(K^{\lambda}_i)_i\|}\cdot\lim_{s\downarrow0}\frac{\|(K^{(s)}_i)_i-(K^{\lambda}_i)_i\|}{s}\\
        &\le \frac{\varepsilon}{\lambda}\cdot \|\tilde{K}^\lambda_i\|^2\\
        &\le C\varepsilon\cdot (\frac{\|(K^{(0)}_i)_i\|}{\lambda}+1).
    \end{align*}
    This calculation is independent of the choice of $i$ and $X_i$. Thus, choosing $\lambda=\|(K^{(0)}_i)_i\|$, we obtain \eqref{eq:variational_ineq}. The case where the derivative \eqref{eq:posotovederi} is negative can be proved similarly.
\end{proof}

\begin{proposition}
 Let $K_1,\ldots,K_k$ be the matrices obtained in Proposition \ref{lem:epsilon} and set $\tilde{K}=\bigoplus^k_{i=1}\tilde{K}_i$, $\tilde{K}_i^*\tilde{K}_i=K_i$ and $L_j=B_j\tilde{K}^*\Pi_\tau \tilde{K}B_j^*$. Let $Z_1,\ldots,Z_k$ be the matrices obtained in Proposition \ref{prop:laglange}. Then for each $i\in[k]$,
   \begin{equation}\label{eq:epsilonfor||}
       \|c_i\mathrm{id}_{E_i}-2\tau(\Pi_\tau^2)_{ii}-Z_i\|<\varepsilon,
   \end{equation}
   where $(\Pi_\tau^2)_{ii}$ is the $\dim E_i\times\dim E_i$ diagonal block of $\Pi_\tau^2$. 
\end{proposition}

\begin{proof}
  Choose $(K_1,\ldots,K_k)$ satisfying Proposition \ref{lem:epsilon}. Fix $i_0$ and a self-adjoint matrix $X_{i_0}$. Set $E_{X_{i_0}}:=\mathrm{diag}(0,\ldots,X_{i_0},\ldots,0)$. By the notation $\tilde{K}^*_i\tilde{K}_i=K_i$, we have 
  \begin{equation}
      \tilde{K}^*_{i_0}e^{sX_{i_0}}\tilde{K}_{i_0}=(e^{sX_{i_0}/2}\tilde{K}_{i_0})^*(e^{sX_{i_0}/2}\tilde{K}_{i_0}).
  \end{equation}
  Hence considering the map $\tilde{K}_{i_0}(s)= e^{sX_{i_0}/2}\tilde{K}_{i_0}$, $\tilde{K}^*_i\tilde{K}_i=K_i$ for $i\neq i_0$ and $\tilde{K}(s)=\tilde{K}_1\oplus\ldots\oplus\tilde{K}_{i_0}(s)\oplus\ldots\oplus\tilde{K}_k$, we obtain 
  \begin{equation}
       \begin{aligned}\label{eq:smapstoK}
      \frac{d}{ds}\bigg|_{s=0}\tilde{K}^*(s)\Pi_{\tau} \tilde{K}(s)&=\frac{d}{ds}\bigg|_{s=0}\tilde{K}^*e^{sE_{{X}_{i_0}}/2}\Pi_{\tau} e^{sE_{{X}_{i_0}}/2}\tilde{K}\\
      &=\frac{1}{2}\tilde{K}^*(E_{{X}_{i_0}}\Pi_{\tau}+\Pi_{\tau}E_{{X}_{i_0}})\tilde{K}.
  \end{aligned} 
  \end{equation}
 
  Using \eqref{eq:differential}, \eqref{eq:smapstoK}, $\det e^{X}=e^{\mathrm{tr}X}$ and $\frac{d}{ds}\log\det X=\mathrm{tr}(X^{-1}\frac{d}{ds}X)$ for invertible matrix $X$ \cite{Lax}, we have
    \begin{equation}
      \begin{aligned} \label{eq:trace}
        \frac{d}{ds}\bigg|_{s=0}\Phi_{\tau}(&K_1,\ldots,\tilde{K}^*_{i_0}e^{sX_{i_0}}\tilde{K}_{i_0},\ldots,K_k)\\
        &=c_i\frac{d}{ds}\bigg|_{s=0}\log\det (\tilde{K}_{i_0}e^{sX_{i_0}}\tilde{K}^*_{i_0})-\sum^m_{j=1}d_j\mathrm{tr}\bigg(L_j^{-1}\frac{d}{ds}\bigg|_{s=0}B_j\tilde{K}^*(s)\Pi_{\tau} \tilde{K}(s)B_j^*)\bigg)\\
        &=c_{i_0}\mathrm{tr}X_{i_0}-\frac{1}{2}\sum^m_{j=1}d_j\mathrm{tr}\bigg(L_j^{-1}B_j\tilde{K}^*(E_{{X}_{i_0}}\Pi_{\tau}+\Pi_{\tau}E_{{X}_{i_0}})\tilde{K}B_j^*)\bigg)\\
        &=c_{i_0}\mathrm{tr}X_{i_0}-\frac{1}{2}\mathrm{tr}\big(E_{{X}_{i_0}}(\Pi_{\tau}M_{\tau}+M_{\tau}\Pi_{\tau})\big)\quad\big(M_{\tau}=\sum^m_{j=1}d_j\tilde{K}B_j^*L_j^{-1}B_j\tilde{K}^*\big)\\
        &=c_{i_0}\mathrm{tr}X_{i_0}-\frac{1}{2}\mathrm{tr}\big({X}_{i_0}\big((\Pi_{\tau}M_{\tau})_{i_0i_0}+(M_{\tau}\Pi_{\tau})_{i_0i_0}\big)\big) .
    \end{aligned}  
    \end{equation}
    
 Now, by \eqref{eq:Laglange}, we have
 \begin{equation}
     M_{\tau}\Pi_\tau=Z\Pi_\tau+2\tau \Pi_\tau^2, \quad \Pi_\tau M_{\tau}=\Pi_\tau Z+2\tau \Pi_\tau^2.
 \end{equation}
 Also by $(Z\Pi_\tau)_{i_0i_0}=Z_i\Pi_{i_0i_0}=Z_{i_0}$ and $(\Pi_\tau Z)_{i_0i_0}=\Pi_{i_0i_0}Z_{i_0}=Z_{i_0}$, each diagonal block is 
 \begin{equation}
      (M_{\tau}\Pi_\tau)_{i_0i_0}=(\Pi_\tau M_{\tau})_{i_0i_0}=Z_{i_0}+2\tau (\Pi_\tau^2)_{i_0i_0}.
 \end{equation}
 Thus, \eqref{eq:trace} means that 
 \begin{equation}
     \frac{d}{ds}\bigg|_{s=0}\Phi_{\tau}(K_1,\ldots,\tilde{K}^*_{i_0}e^{sX_{i_0}}\tilde{K}_{i_0},\ldots,K_k)=\mathrm{tr}\bigg({X}_{i_0}\big(c_{i_0}I_{\dim E_{i_0}}-2\tau(\Pi_\tau^2)_{i_0i_0}-Z_{i_0}\big)\bigg).
 \end{equation}

\noindent
Thus we can use Proposition \ref{lem:epsilon} to obtain \eqref{eq:epsilonfor||}.
 
\end{proof}

\subsection{Proof of Theorem \ref{thm:Main}}
Now we give a proof of Theorem \ref{thm:Main}.

\begin{proof}[Proof of Theorem \ref{thm:Main}]
Set $c_*:=\min_{i\in[k]}c_i>0$. Fix $0<\varepsilon<c_*$, and choose $\tilde{K}$, $\Pi_{\tau}$ and $L_j=B_j\tilde{K}^*\Pi_\tau \tilde{K}B_j^*$ that satisfy Proposition \ref{lem:epsilon} and Proposition \ref{prop:laglange} in the case of $\varepsilon/4$. 
 
Set \[B_j'=\big(\frac{c_*}{c_*+\varepsilon}\big)^{1/2}L_j^{-1/2}B_j\tilde{K}^*.\]Then 
\begin{equation}
    B_j'\Pi_{\tau}B_j'^*=\frac{c_*}{c_*+\varepsilon}L_j^{-1/2}L_jL_j^{-1/2}=\mathrm{id}_{E^j}-\frac{\varepsilon}{c_*+\varepsilon}\mathrm{id}_{E^j}.
\end{equation}
Also, using \eqref{eq:Laglange} and \eqref{eq:epsilonfor||}, we have
\begin{equation}
    \sum^m_{j=1}d_jB_j'^*B_j'=\frac{c_*}{c_*+\varepsilon}M_{\tau}=\frac{c_*}{c_*+\varepsilon}(Z+2\tau\Pi_{\tau})\le\frac{c_*}{c_*+\varepsilon}\Big(\Lambda_{\mathbf{c}}+\frac{\varepsilon}{4}\mathrm{id}_{E_0}+2\tau\big(\Pi_{\tau}-\small\bigoplus_i(\Pi_\tau^2)_{ii}\big)\Big).
\end{equation}
Here we used the equivalence of norms on finite-dimensional space. We remark $\Pi_{\tau}$ is uniformly bounded in $\tau$. Hence choosing $\tau$ to satisfy $2\tau\big(\Pi_{\tau}-\small\bigoplus_i(\Pi_\tau^2)_{ii}\big)\le \varepsilon/4\cdot \mathrm{id}_{E_0}$,
\begin{equation}
    \sum^m_{j=1}d_jB_j'^*B_j'\le \frac{c_*}{c_*+\varepsilon}(\Lambda_{\mathbf{c}}+\frac{\varepsilon}{2}\cdot\mathrm{id}_{E_0})\le\Lambda_{\mathbf{c}}.
\end{equation}
Therefore, we obtain \eqref{eq:seqproj}.

\end{proof}

\noindent
\textbf{Acknowledgements.} The author would like to express her sincere gratitude to her advisor, Professor Neal Bez, for his guidance and support throughout this work. His valuable comments and suggestions improved the clarity and readability of the manuscript. The author  thanks Genki Wakabayashi for helpful conversations which provided inspiration for the approach taken in the proof of Proposition 3.2. She is also  grateful to Professor Yoshihiro Tonegawa, her former advisor, whose guidance and influence led her to ideas that ultimately brought Ekeland's variational principle to her attention in the proof of Proposition 3.2.

\medskip

\section*{Appendix.} Here we present a proof of ubiquity in the framework of quiver capacity, thus giving a new proof of \cite[Corollary 14]{CDN}. Both the forward-reverse Brascamp--Lieb constant and capacity are quantities associated with data consisting of linear maps between two collections of Euclidean spaces, and both frameworks include the classical Brascamp--Lieb inequality as a special case when $k=1$. Nevertheless, the two are distinct notions, and the maximization over couplings in the forward-reverse setting prevents an explicit computation of derivatives, making the analytic proof of ubiquity more involved.

\medskip
Let $\mathbf{d}=(d_1,\ldots,d_m)$ satisfy $\sum^k_{i=1}\mathrm{dim}E_i=\sum^m_{j=1}d_j\dim E^j$. For each pair $(i,j)$, let $\mathbf{A}_{ij}$ be a finite index set and denote \[\mathbf{V}:=(V_a\in E_i\otimes E^j : a\in {\mathbf{A}}_{ij},i\in [k],j\in [m]).\]

\noindent
The pair $(\mathbf{V},\mathbf{d})$ is called a \textit{quiver datum}, and its \textit{capacity} is defined as follows:  
\medskip
\begin{equation}\label{eq:capp}
    \mathrm{Cap}(\mathbf{V},\mathbf{d}):=\inf_{(U_j\in \mathsf{PD}_{n_j})_{j\in[m]} }\mathrm{Cap}(\mathbf{V},\mathbf{d};(U_j)_{j\in[m]}),\end{equation}where
\begin{equation}\label{eq:cap}
    \mathrm{Cap}(\mathbf{V},\mathbf{d};(U_j)_{j\in[m]}):=\frac{\prod^k_{i=1} \det\big(\sum^m_{j=1}d_j\sum_{a\in\mathbf{A}_{ij}} V^*_aU_jV_a\big)}{\prod^n_{j=1}(\det U_j)^{d_{i}}}.
\end{equation}
\noindent
In this setting, geometric data are defined by equations
\[\sum^{m}_{j=1}d_{j}\sum_{a\in {\mathbf{A}}_{ij}}V_a^*V_a = id_{E_i}  \text{ and }\sum^{k}_{i=1}\sum_{a\in {\mathbf{A}}_{ij}}V_aV_a^* = id_{E^j}\]
for each $i\in [k]$ and $j\in [m]$ while the equivalence relation between quiver data is defined by $V_a'=C_jV_a'C_i$ for every $a\in\mathbf{A}_{ij}$, $i\in[k]$ and $j\in[m]$. With these definitions, the ubiquity result can be formulated as follows.
\begin{theorem}[\cite{CDN}]\label{theorem:ubicap}
      Let $(\mathbf{V}, \mathbf{d})$  be a quiver datum such that $\mathrm{Cap}(\mathbf{V}, \mathbf{d})>0$. Then there exists some tuple of matrices $\mathbf{G}$ such that $(\mathbf{G}, \mathbf{d})$ is geometric and for any $\varepsilon>0$, there is some tuple of matrices $\mathbf{V}'$ which is equivalent to $\mathbf{V}$ such that \begin{equation}\label{eq:Cap-ubiquity}\| \mathbf{G}-\mathbf{V}' \|< \varepsilon.\end{equation}
\end{theorem}

\noindent
In our new analytic proof of Theorem \ref{theorem:ubicap}, we consider the function \[f\big((t_{js})_{j\in[m],s\in[n_j]},(R_j)_{j\in[m]}\big):=-\log\mathrm{Cap}(\mathbf{V},\mathbf{d};(R_j\mathrm{diag}(e^{t_{j1}},\ldots,e^{t_{jn_j}})R_j^*)_{j\in[m]})\] based on the diagonalizing of positive definite matrices. The function $f$ is bounded from above when $\mathrm{Cap}(\mathbf{V},\mathbf{d})>0$. Here we use the following lemma, which follows immediately from \cite[Lemma 3.4]{DH}, instead of Proposition \ref{lem:epsilon}.

\begin{lemma}\label{lemma:epsilonforcap}
  For each $\varepsilon>0$, there exists a pair $(t_{js})_{j\in[m],s\in[n_j]},(R_j)_{j\in[m]}$ such that  \[\Big| \frac{d}{dt_{js}}f\big((t_{js})_{j\in[m],s\in[n_j]},(R_j)_{j\in[m]}\big) \Big|< \varepsilon\]for all $j\in[m]$ and $s\in[n_j]$.
\end{lemma}
As can be seen from \eqref{eq:cap}, the derivatives of $f$ in Lemma \ref{lemma:epsilonforcap} can be computed directly. Using $(t_{js})_{j\in[m],s\in[n_j]},(R_j)_{j\in[m]}$ in Lemma \ref{lemma:epsilonforcap}, we construct the tuple of linear maps $\mathbf{V}'$ as follows:
\[V_a'=\begin{pmatrix}
  e^{\frac{t_{j1}}{2}}                                                \
            &        & \text{\huge{0}}   \\
         &  \ddots               &                      \\
          \text{\huge{0}} &        &            
                                         e^{\frac{t_{jn_j}}{2}}
\end{pmatrix}R_j^*V_aM_i^*\quad\text{for each $a\in \mathbf{A}_{ij}$},\]where
\[M_i^*M_i=X_i^{-1}\text{ and } X_i=\sum_{j\in [m]}d_j\sum_{a\in {\mathbf{A}}_{ij}}V_a^*R_j\begin{pmatrix}
  e^{{t_{j1}}}                                                \
            &        & \text{\huge{0}}   \\
         &  \ddots               &                      \\
          \text{\huge{0}} &        &            
                                         e^{t_{jn_j}}
\end{pmatrix}R_j^*V_a.\]
Then by the calculation of derivatives in Lemma \ref{lemma:epsilonforcap}, we see that $\mathbf{V}'$ satisfies \[\sum^{m}_{j=1}d_{j}\sum_{a\in {\mathbf{A}}_{ij}}(V_a')^*V'_a =id_{E_i} \text{ and }\Big\|\sum^{k}_{i=1}\sum_{a\in {\mathbf{A}}_{ij}}V'_a(V_a')^* - id_{E^j}\Big\|<\varepsilon.\]

\medskip
\medskip
\noindent
\textbf{Declaration of AI use.} During the preparation of this manuscript, the author used ChatGPT's Deep Research feature to assist with the literature search, which helped her quickly locate Danskin's theorem and relevant references. In addition, she used ChatGPT as an assistive tool to improve English grammar.

 \medskip
\medskip
\noindent
\small\textsc{Graduate School of Mathematical Sciences, The University of Tokyo, 3-8-1 Komaba, Meguro-ku, Tokyo 153-8914, Japan}

\noindent
\textit{Email address}: \texttt{mireille-labeille@g.ecc.u-tokyo.ac.jp}
\end{document}